\documentclass[english,a4paper,11pt,draft]{article}
\usepackage{amsmath,amssymb,amsfonts,latexsym, amsthm, mathrsfs}
\usepackage[active]{srcltx}
\catcode`\@=11
\@addtoreset{equation}{section}

\catcode`\@=12

\newtheorem{Theorem}{Theorem}[section]
\newtheorem{Lemma}[Theorem]{Lemma}
\newtheorem{Proposition}[Theorem]{Proposition}

\newtheorem{Definition}{Definition}[section]

\newcommand{\Rn}{{\mathbb R}^{N}}
\newcommand{\N}{\mathbb{N}}

\newcommand{\ba} {\beta}
\newcommand{\ga} {\gamma}

\newcommand{\De} {\Delta}
\newcommand{\la} {\lambda}

\newcommand{\pa}{\partial}
\newcommand{\na} {\nabla}
\newcommand{\va} {\varphi}
\newcommand{\Ome}{\Omega}

\newcommand{\no}{\nonumber}
\newcommand{\bdw}{\partial\Omega}

\newcommand{\R}{\mathbb{R}}

\newcommand{\per}{\phi_{\epsilon,R}}
\newcommand{\deb}{\rightharpoonup}
\newcommand{\starstar}{2^{*\!*} }
\newcommand{\Huno}{H^2_0(\Omega)}

\newcommand{\Iom}{\int_{\Omega}}
\newcommand{\qb}{q_{\beta}}

\begin{document}

\title{Caffarelli-Kohn-Nirenberg type equations of fourth order with the critical exponent and Rellich potential}

\author{Mousomi Bhakta
\footnote
{Department of Mathematics, Indian Institute of Science Education and Research, Pune, India.
Email: {mousomi@iiserpune.ac.in}}}

\date{}

\maketitle

\begin{abstract}
\noindent
\footnotesize We study the existence/nonexistence of positive solutions of 
$$
{\Delta^2u-\mu\frac{u}{|x|^4}=\frac{|u|^{q_{\ba}-2}u}{|x|^{\beta}}\quad\textrm{in $\Omega$,}}
$$
when $\Omega$ is a bounded domain and $N\geq 5$, $q_{\ba}=\frac{2(N-\ba)}{N-4}$, $0\leq \ba<4$ and $0\leq\mu<\big(\frac{N(N-4)}{4}\big)^2$. We prove the nonexistence result when $\Omega$ is an open subset of $\Rn$ which is star shaped with respect to the origin. We also study the existence of positive solutions  when $\Omega$ is a smooth bounded domain with non trivial topology and  $\ba=0$, $\mu\in(0,\mu_0)$, for certain $\mu_0<\big(\frac{N(N-4)}{4}\big)^2$ and $N\geq 8$. Different behavior of PS sequences have been obtained depending whether $\ba=0$ or $\ba>0$.
%
\bigskip

\noindent
\textbf{Keywords:} {Caffarelli-Kohn-Nirenberg type equations, Palais-Smale decomposition, nonexistence, Pohozaev identity, nontrivial topology, contractible domain, fourth order equation with singularity.}
\medskip

\noindent
\textit{2010 Mathematics Subject Classification:} {31B30, 35J35, 35J75, 35J91.}
\end{abstract}

\section{Introduction}

In this article  we  study  the singular semilinear fourth order elliptic problem:
\begin{equation}
\label{eq:problem}
\begin{cases}
  \Delta^2u-\mu\frac{u}{|x|^4}=\frac{|u|^{q_{\ba}-2}u}{|x|^{\beta}} &\textrm{in $\Omega$},\\
u\in  H^2_0(\Omega),\\
u>0 \quad\text{in}\ \Omega,
\end{cases}
\end{equation}
where
$\Delta^2 u=\Delta(\Delta u)$, $\Omega$ is a smooth bounded domain and
\begin{equation}
\label{eq:ass_q_mu}
 N\geq 5, \ q_{\ba}=\frac{2(N-\ba)}{N-4} , \  0\leq\ba<4 \ \text{and}\  0\leq\mu<\bar\mu=\displaystyle\left(\frac{N(N-4)}{4}\right)^2.
\end{equation}

Semilinear elliptic equations with biharmonic operator arise in continuum mechanics, biophysics, differential geometry. In particular in  the modeling of thin elastic plates, clamped plates and in the study of the Paneitz-Branson equation and the Willmore equation  (see \cite{GGS} and the references therein for more details). 

\begin{Definition}
We say $u\in H^2_0(\Omega)$ is a solution to \eqref{eq:problem} if $u>0$ in $\Omega$ and satisfies
 $$\int_{\Omega}\displaystyle\left[\Delta u\Delta\phi-\mu\frac{u\phi}{|x|^4}\right]dx=\int_{\Omega}\frac{|u|^{q_{\ba}-2}u\phi}{|x|^{\ba}}dx \quad \forall\ \phi\in H^2_0(\Omega).$$
\end{Definition}

\vspace{2mm}

Equivalently $u$ is a critical point of the functional 
\begin{equation}\label{I-mu}
I_{\mu}(u)=\frac{1}{2}\int_{\Omega}\displaystyle\left[|\Delta u|^2-\mu\frac{|u|^2}{|x|^4}\right]dx-\frac{1}{q_{\ba}}\int_{\Omega}\frac{|u|^{q_{\ba}}}{|x|^{\ba}}dx  \quad u\in H^2_0(\Omega).
\end{equation}

$I_{\mu}$ is a well defined $C^1$ functional in $H^2_0(\Omega)$, thanks to the following Rellich inequality (\cite{Rel54}, \cite{Rel69}) :
\begin{equation}\label{Rellich}
\int_{\Rn}|\De u|^2dx\geq \displaystyle\bar\mu\int_{\Rn}|x|^{-4}|u|^2 dx\quad\forall \   \  u\in \mathcal C^{\infty}_{0}(\R^N),
\end{equation}
and the  Cafferelli-Kohn-Nirenberg (CKN) inequality of fourth order ( \cite{BM}, \cite{CKN}, \cite{CM}) :
\begin{equation}\label{CKN}
\int_{\Rn}|\De u|^{2}dx\ge C\left(\int_{\Rn}
|x|^{-\beta}|u|^{q_{\ba}}dx\right)^{2/\qb}\quad \forall \ u\in \mathcal C^{\infty}_0(\Rn), 
\end{equation}
where $C=C(N, \ba)>0$. Rellich inequality is generalization of the following Hardy inequality:
\begin{equation}\label{Hardy}
\displaystyle\left(\frac{N-2}{2}\right)^2\int_{\Rn}\frac{|u|^2}{|x|^2}dx\leq\int_{\Rn}|\nabla u|^2dx \quad\forall\ u\in C^{\infty}_0(\Rn).
\end{equation}

\textit{Remark}: It is well known that when $\Omega$ is a smooth bounded domain Hardy inequality holds for every $u\in H^1_0(\Omega)$ but the best constant $\displaystyle\left(\frac{N-2}{2}\right)^2$ is never achieved. 

\vspace{2mm}

From literature we know that the usual norm in $H^2(\Omega)$ is $\displaystyle\left(\int_{\Omega}\sum_{0\leq|\alpha|\leq 2}|D^{\alpha}u|^2 dx\right)^\frac{1}{2}$. Thanks to interpolation theory, one can neglect intermediate derivates and see that 
\begin{equation}\label{norm-1}
||u||_{H^2(\Omega)}=\displaystyle\left(\Iom|u|^2 dx+\Iom|D^2u|^2 dx\right)^\frac{1}{2}
\end{equation}
 defines a norm which is equivalent to the usual norm in $H^2(\Omega)$ (see \cite{A}). As $\Omega$ is a smooth bounded domain and $H^2_0(\Omega)$ is the closure of $C_0^{\infty}(\Omega)$ w.r.t. the norm in  $H^2(\Omega)$, invoking \cite[Theorem 2.2]{GGS} we find that
\begin{equation}\label{norm-2}
||u||_{H^2_0(\Omega)}=\displaystyle\left(\Iom|\Delta u|^2 dx\right)^\frac{1}{2}
\end{equation}
defines an quivalent norm to \eqref{norm-1}. Now onwards we will consider $H^2_0(\Omega)$ endowed with the norm defined in \eqref{norm-2}.

We also note that as $\mu<\bar\mu$, applying Rellich inequality it is not diffcicult to check that  $\displaystyle\left(\int_{\Omega}\displaystyle\left[|\Delta u|^2-\mu\frac{|u|^2}{|x|^4}\right]dx\right)^\frac{1}{2}$ becomes an equivalent norm to \eqref{norm-2}.  When $\ba=0$, \eqref{CKN} is the Sobolev inequality by using the equivalence between norms explained above.. When $\ba=0$ we denote $q_0$ as $\starstar:=\frac{2N}{N-4}$.
\\
We define
\begin{equation}\label{S-mu}
S_{\mu}^{\ba}:=\inf_{u\in\Huno, u\not\equiv 0}\frac{\displaystyle\int_{\Omega}\left[|\Delta u|^2-\mu\frac{|u|^2}{|x|^4}\right]dx}{\displaystyle\left(\int_{\Omega}\frac{|u|^{\qb}}{|x|^{\ba}}dx\right)^\frac{2}{\qb}}, 
\end{equation}
where $0\leq\mu<\bar\mu$. Using \eqref{Rellich} and \eqref{CKN} we see that $S_{\mu}^{\ba}>0$.

\vspace{2mm}

Problem \eqref{eq:problem} is closely related to some intensively studied topics. For instance, when $\mu=0$, the fourth order differential
equation in \eqref{eq:problem} is equivalent to
\begin{equation}
\begin{cases}
-\Delta u = v,\\
-\Delta v = \frac{|u|^{\qb-2}u}{|x|^{\ba}},
\end{cases}
\end{equation}
which can be regarded as a H\'{e}non-Lane-Emden type system; see for instance (\cite{CR}, \cite{FG}, \cite{Soup}) and the references therein.
The second-order version of \eqref{eq:problem}, namely,
\begin{equation}
\begin{cases}
-\Delta u = \frac{u}{|x|^2}+\frac{|u|^{p_t-2}u}{|x|^t} \quad\text{in}\ \Omega,\\
u\in H^1_0(\Omega),
\end{cases}
\end{equation}
where $p_t=\frac{2(N-t)}{N-2}$, has been widely studied in recent years in bounded domains and in the whole space $\Rn$ for the case $t=0$ or $0<t<2$ (see for instance \cite{BS}, \cite{FG1}, \cite{GM}, \cite{HS}, \cite{Mu}, \cite{Smets}, \cite{Ter} and the references therein). In a very recent work \cite{PP}, a problem related to \eqref{eq:problem} when $\ba=0$ with Navier boundary condition  has been studied by  Pérez-Llanos and Primo. More precisely, in \cite{PP} the authors have studied the optimal power $p$ for existence/nonexistence of distributional solutions of the following problem:
\begin{equation}
\begin{cases}
\Delta^2 u = \la\frac{u}{|x|^4}+u^p+cf \quad\text{in}\quad\Omega,\\
u>0  \quad\text{in}\quad\Omega,\\
u=\Delta u=0  \quad\text{on}\quad\Omega,
\end{cases}
\end{equation}
in a smooth and bounded domain such that $0\in\Omega$. 

\vspace{2mm}

When $\Omega=\Rn$, the existence of solution to \eqref{eq:problem} has been studied in \cite{BM}. In a bounded domain the problem \eqref{eq:problem} does not have a solution in general, due to the critical
nature of the equation.  The main difficulty here is singularity at the origin. The sufficient regularity of the test function is guaranteed with the use of cut-off fuctions. In addition, the passages to the limit are delicate and clever estimates are needed. The nonexistence of solution to \eqref{eq:problem} in bounded domain is due to the lack of compactness of the functional $I_{\mu}$, given in \eqref{I-mu}, as a result of concentration phenomenon. We analyze this noncompactness in Theorem 3.1, showing that concentration takes place along a single profile when $\ba > 0$ whereas if $\ba=0$, concentration occurs along two different profiles. Thanks to this behavior we are able to prove an existence result  in a non-contractible bounded domain when $\ba=0$ and the parameter $\mu\in(0,\mu_0)$ for some suitable $\mu_0$. 

\vspace{3mm}

This paper is organized as follows: We prove a Pohozaev type nonexistence result in Section 2. Section 3 is devoted to the study of PS sequences when $\ba>0$ and $\ba=0$.  Finally, in Section 4 the existence result when $\ba=0$ in suitable bounded domains is given.

\vspace{2mm}

{\bf Notations :} We denote by $H^2(\Omega)$ the usual Sobolev space $W^{2,2}(\Omega)$ and by $D^{2,2}(\Ome)$ the closure of $C_0^\infty(\Ome)$ with respect to the norm $\left(\Iom |\Delta u|^2\right)^{\frac{1}{2}}$. Similarly $D^{2,2}(\Rn)$ can be defined. By $D^{1,2}(\Rn)$ we denote the closure of $C^{\infty}_0(\Rn)$ with respect to the norm $\left(\Iom |\na u|^2\right)^{\frac{1}{2}}$. We note here that, thanks to the previous equivalence relation between the norms, if $u\in D^{2,2}(\Rn)$ then $\na u\in D^{1,2}(\Rn)$. When $\Omega$ is unbounded, for instance $\Omega=\Rn$, it is known that $H^2(\Rn)=H^2_0(\Rn)\subset D^{2,2}(\Rn)$ but the converse inclusion fails. On the other hand, if $\Omega$ is a smooth and bounded domain, then $D^{2,2}(\Omega)=H^2_0(\Omega)$, thanks to \cite[Theorem 2.1]{GGS}.
If $u\in\Huno$ then $u=|\na u|=0$ on $\bdw$. $C$ and $c$ will be general constants which may vary from line to line. $B_r(a)$ will denote the ball of radius $r$, centered at $a$.

\section{Nonexistence result}
In this section we will present the nonexistence result whose proof is based on the Pohozaev identity. The
difficulty in applying this identity is because of the singularity at the origin but we overcome this by using Hardy inequality (\eqref{Hardy})  and Rellich inequality. More precisely, to make the test function smooth, we introduce cut-off functions and pass to the limit with the help of Hardy, Rellich and Sobolev inequalities. 

\begin{Theorem}
Let $\Omega$ be an open subset of $\Rn$ with smooth boundary and star shaped with respect to origin. Then the problem 
\begin{equation}
\label{eq:problem2}
\begin{cases}
  \Delta^2u-\mu\frac{u}{|x|^4}=\frac{|u|^{q_{\ba}-2}u}{|x|^{\beta}} &\textrm{in $\Omega$},\\
u\in H^2_0(\Omega)
\end{cases}
\end{equation}
has a nontrivial nonnegative solution only if $\Omega=\Rn$.
\end{Theorem}

\begin{proof}
For $\epsilon>0$ and $R>0$, we define $\phi_{\epsilon,R}(x)=\phi_{\epsilon}(x)\psi_{R}(x)$ where 
$\phi_{\epsilon}(x)=\phi(\frac{|x|}{\epsilon})$ and $\psi_R(x)=\psi(\frac{|x|}{R})$, $\phi$ and $\psi$ are smooth functions in $\R$ with the properties $0\leq\phi,\psi\leq 1$, with supports of $\phi$ and $\psi$ in $(1,\infty)$ and $(-\infty, 2)$ respectively and $\phi(t)=1$ for $t\geq 2$, and $\psi(t)=1$ for $t\leq 1$.

Assume that \eqref{eq:problem2} has a nonnegative nontrivial solution $u$ in $\Omega$ and $\Omega\not=\Rn$. It's known that $u$ is smooth away from the origin (see \cite[page 235-236]{GGS}) and
hence $(x\cdot\na u)\phi_{\epsilon,R}\in C^3_c(\bar\Omega)$. Multiplying Eq.\eqref{eq:problem2} by this test function  we obtain

\begin{eqnarray}\label{non-1}
\Iom\Delta^2 u(x\cdot\na u)\phi_{\epsilon,R}dx
&=&\mu\Iom\frac{u(x\cdot\na u)}{|x|^4}\phi_{\epsilon,R}dx+\Iom\frac{|u|^{q_{\ba}-2}u}{|x|^{\ba}}(x\cdot\na u)\phi_{\epsilon,R}dx\no\\
&=& I_1+I_2.
\end{eqnarray}
First we will simplify the RHS of \eqref{non-1}. Here we observe that $u\in\Huno$ implies that $u=|\na u|=0$ on $\pa\Omega$.
\begin{eqnarray}\label{non-2}
I_1&=&\mu\Iom\frac{u(x\cdot\na u)}{|x|^4}\phi_{\epsilon,R}dx=\frac{\mu}{2}\Iom\na(|u|^2)\cdot\frac{x}{|x|^4}\phi_{\epsilon,R}dx\no\\
&=&-\displaystyle\sum_{i=1}^N\frac{\mu}{2}\Iom|u|^2\left[\frac{x_i(\phi_{\epsilon,R})_{x_i}}{|x|^4}+\frac{\phi_{\epsilon,R}}{|x|^4}\right]dx-\sum_{i=1}^N\frac{\mu}{2}\Iom|u|^2 x_i\frac{\pa}{\pa x_{i}}\displaystyle\left(\frac{1}{|x|^4}\right)\phi_{\epsilon,R}dx\no\\
&=&-\mu\displaystyle\left(\frac{N-4}{2}\right)\Iom\frac{|u|^2}{|x|^4}\phi_{\epsilon,R}dx-\frac{\mu}{2}\Iom\frac{|u|^2}{|x|^4}(x\cdot\na\phi_{\epsilon,R})dx.
\end{eqnarray}
Note that $\na\phi_{\epsilon}$ and $\na\psi_{R}$ have supports respectively in $\{\epsilon<|x|<2\epsilon\}$  and $\{R<|x|<2R\}$. Therefore in the support of the 2nd integral on the RHS of the above relation $|x\cdot\na\phi_{\epsilon,R}|\leq C$ and thus by dominated covergence theorem we see, 
\begin{equation}\label{no-I_1}
\lim_{R\to\infty}\lim_{\epsilon\to 0}I_1=-\mu\displaystyle\left(\frac{N-4}{2}\right)\Iom\frac{|u|^2}{|x|^4}dx.
\end{equation}
Similarly we can compute $I_2$ and get
\begin{eqnarray}\label{non-3}
I_2 &=&\Iom\frac{|u|^{q_{\ba}-2}u}{|x|^{\ba}}(x\cdot\na u)\phi_{\epsilon,R}dx\no\\
&=&-\displaystyle\left(\frac{N-4}{2}\right)\Iom\frac{|u|^{q_{\ba}}}{|x|^{\ba}}\phi_{\epsilon,R}dx-\frac{1}{\qb}\Iom\frac{|u|^{q_{\ba}}}{|x|^{\ba}}(x\cdot\na\phi_{\epsilon,R})dx.
\end{eqnarray}
Consequently, as before we have
\begin{equation}\label{no-I_2}
\lim_{R\to\infty}\lim_{\epsilon\to 0}I_2=-\displaystyle\left(\frac{N-4}{2}\right)\Iom\frac{|u|^{q_{\ba}}}{|x|^{\ba}}dx,
\end{equation}
Now we compute the LHS of \eqref{non-1}. 
\begin{eqnarray}\label{no-4}
\Iom(\Delta^2 u)(x\cdot\na u)\phi_{\epsilon,R}dx&=&-\displaystyle\sum_{i,j=1}^{N}\Iom\left[(\Delta u)_{x_i}(x_j u_{x_j})_{x_i}\phi_{\epsilon,R}+(\Delta u)_{x_i}(x_j u_{x_j})(\phi_{\epsilon,R})_{x_i}\right]dx\no\\
&=&-\sum_{i=1}^{N}\Iom(\Delta u)_{x_i}u_{x_i}\per dx-\sum_{i,j=1}^{N}\Iom(\Delta u)_{x_i}x_j u_{x_i x_j}\per dx\no\\
&+&\sum_{i,j=1}^{N}\Iom(\Delta u)\big(\delta_{ij}u_{x_j}+x_j u_{x_jx_i}\big)(\phi_{\epsilon, R})_{x_i}dx+\Iom\Delta u\Delta \per(x\cdot\na u)dx\no
\\
&=&\Iom|\Delta u|^2\per dx+\Iom\Delta u(\na u\cdot\na\per)dx+\Iom|\Delta u|^2\per dx\no\\
&+&\Iom\Delta u\big[\na(\Delta u)\cdot x\big]\per dx+\sum_{i=1}^{N}\Iom\Delta u(\na u_{x_i}\cdot x)(\per)_{x_i} dx\no\\
&-&\sum_{i,j=1}^N \int_{\pa\Omega}(\Delta u)u_{x_ix_j}\per(x_j\nu_i)dS+\Iom\Delta u(\na u\cdot\na\per)dx\no\\
&+&\sum_{j=1}^N\Iom(\Delta u)x_j(\na\per\cdot\na u_{x_j})dx+\Iom\Delta u\Delta\per(x\cdot\na u)dx.\end{eqnarray}
As a result,
\begin{eqnarray}\label{no-5}
\Iom(\Delta^2 u)(x\cdot\na u)\phi_{\epsilon,R} dx&=&
2\Iom|\Delta u|^2\per dx+\Iom\Delta u\big[\na(\Delta u)\cdot x\big]\per dx\no\\
&-&\sum_{i,j=1}^N\int_{\pa\Omega}(\Delta u)u_{x_ix_j}\per(x_j\nu_i) dS+I,
\end{eqnarray}
where $I=2I_3+I_4+I_5+I_6$ and\\
$I_3=\displaystyle\Iom\Delta u(\na u\cdot\na\per) dx$,\\
$I_4=\displaystyle\sum_{i=1}^{N}\Iom\Delta u(\na u_{x_i}\cdot x)(\per)_{x_i} dx$,\\
$I_5=\displaystyle\sum_{j=1}^N\Iom(\Delta u)x_j(\na\per\cdot\na u_{x_j})dx$,\\
$I_6=\displaystyle\Iom\Delta u\Delta\per(x\cdot\na u)dx$.\\
A simple computation shows that 2nd term on RHS of \eqref{no-5} equals
\begin{eqnarray}\label{no-6}
\Iom\Delta u\big[\na(\Delta u)\cdot x\big]\per dx &=&-\frac{N}{2}\Iom|\Delta u|^2\per dx-\frac{1}{2}\Iom|\Delta u|^2(x\cdot\na\per) dx\no\\
&+&\frac{1}{2}\int_{\pa\Omega}|\Delta u|^2(x\cdot\nu)\per dS.
\end{eqnarray}
Since $u_{x_i}=0$ on $\bdw$, we have on $\bdw$,
$$\na(u_{x_i})=\frac{\pa}{\pa\nu}(u_{x_i})\cdot\nu=\sum_{l=1}^N(u_{x_lx_i}\nu_l)\cdot\nu,$$ i.e.,  $u_{x_ix_j}=\sum_{l=1}^{N}u_{x_lx_i}\nu_l\nu_j.$
Hence we compute the 3rd term on the RHS of \eqref{no-5} as follows
\begin{eqnarray}\label{no-7}
-\int_{\pa\Omega}(\Delta u)u_{x_ix_j}(x_j\nu_i)\per &=&-\int_{\pa\Omega}(\Delta u)u_{x_lx_i}\nu_l\nu_j\nu_i x_j\per =-\int_{\pa\Omega}(\Delta u)u_{x_lx_i}\nu_l\nu_i(x\cdot\nu)\per \no\\
&=&-\int_{\bdw}\Delta u\frac{\pa^2u}{\pa\nu^2}(x\cdot\nu)\per =-\int_{\bdw}|\Delta u|^2(x\cdot\nu)\per ,
\end{eqnarray}
where we again used the fact that $u_{x_i}=0$ on $\bdw$.
Thus combining \eqref{no-5},\eqref{no-6} and \eqref{no-7} we obtain
\begin{eqnarray}\label{no-lhs}
\Iom(\Delta^2 u)(x\cdot\na u)\phi_{\epsilon,R} dx&=&-\displaystyle\left(\frac{N-4}{2}\right)\Iom|\Delta u|^2\per dx-\frac{1}{2}\int_{\bdw}|\Delta u|^2(x\cdot\nu)\per dS\no\\
&-&\frac{1}{2}\Iom|\Delta u|^2(x\cdot\na\per)dx+I.
\end{eqnarray}
As before 
\begin{equation}\label{no-8}
\lim_{R\to\infty}\lim_{\epsilon\to 0}\frac{1}{2}\Iom|\Delta u|^2(x\cdot\na\per)dx=0.
\end{equation}
Next, we will prove that $\lim_{R\to\infty}\lim_{\epsilon\to 0}I=0$ by estimating  $I_3,I_4,I_5$ and $I_6$.

\begin{eqnarray}\label{no-p9}
I_6 &=&\Iom\Delta u(x\cdot\na u)\Delta\per=\Iom\Delta u(x\cdot\na u)[\psi_R\Delta\phi_{\epsilon}+2\na\phi\na\psi+\phi_{\epsilon}\Delta\psi_R]\no\\
&\leq&\int_{\Omega\cap\{\epsilon\leq|x|\leq2\epsilon\}}|\Delta u||\na u||x|\displaystyle\left(\frac{c}{\epsilon^2}+\frac{c}{\epsilon|x|}\right)+\int_{\Omega\cap\{R\leq|x|\leq2R\}}|\Delta u||\na u||x|\displaystyle\left(\frac{c}{R^2}+\frac{c}{R|x|}\right)\no\\
&+&\int_{\Omega\cap\{\epsilon\leq|x|\leq2\epsilon\}\cap\{R\leq|x|\leq 2R\}}|\Delta u||\na u||x|\frac{c}{\epsilon R}\no\\
& \leq & C\int_{\Omega\cap\{\epsilon\leq|x|\leq 2\epsilon\}}|\Delta u|\frac{|\na u|}{\epsilon}+C\int_{\Omega\cap\{R\leq|x|\leq 2R\}}|\Delta u|\frac{|\na u|}{R}\no\\
&\leq&C\displaystyle\left(\int_{\Omega\cap\{\epsilon\leq|x|\leq2\epsilon\}}|\Delta u|^2\right)^\frac{1}{2}\left(\int_{\Omega\cap\{\epsilon\leq|x|\leq2\epsilon\}}\frac{|\na u|^2}{|x|^2}\right)^\frac{1}{2}\no\\
&+&C\displaystyle\left(\int_{\Omega\cap\{R\leq|x|\leq 2R\}}|\Delta u|^2\right)^\frac{1}{2}\left(\int_{\Omega\cap\{R\leq|x|\leq 2R\}}\frac{|\na u|^2}{|x|^2}\right)^\frac{1}{2}. 
\end{eqnarray}
As $u\in\Huno$, we have $\na u\in H^1_0(\Omega)$. Therefore, invoking Hardy inequality \eqref{Hardy} and the remark following that, 
we have $$\int_{\Omega\cap\{\epsilon\leq|x|\leq2\epsilon\}}\frac{|\na u|^2}{|x|^2} dx\leq\Iom\frac{|\na u|^2}{|x|^2} dx\leq\displaystyle\left(\frac{2}{N-2}\right)^2\Iom|D^2 u|^2 dx\leq C.$$
Similarly, it can be shown that $\displaystyle\int_{\Omega\cap\{R\leq|x|\leq 2R\}}\frac{|\na u|^2}{|x|^2} dx\leq C$. As a result, from \eqref{no-p9} it follows
\begin{equation}\label{no-9}
\lim_{R\to\infty}\lim_{\epsilon\to 0}I_6=C\displaystyle\left(\int_{\Omega\cap\{\epsilon\leq|x|\leq2\epsilon\}}|\Delta u|^2 dx\right)^\frac{1}{2}+C\left(\int_{\Omega\cap\{R\leq|x|\leq 2R\}}|\Delta u|^2 dx\right)^\frac{1}{2}=0.
\end{equation}
Likewise, it can be proved that
\begin{equation}\label{no-10}
\lim_{R\to\infty}\lim_{\epsilon\to 0}I_3=0.
\end{equation}
\begin{eqnarray}\label{no-p11}
I_4=\sum_{i=1}^{N}\Iom\Delta u(\na u_{x_i}\cdot x)(\per)_{x_i}dx\leq C\int_{\Omega\cap\{\epsilon\leq|x|\leq2\epsilon\}}|\Delta u||D^2 u|\frac{|x|}{\epsilon}dx\no\\
+C\int_{\Omega\cap\{R\leq|x|\leq2R\}}|\Delta u||D^2 u|\frac{|x|}{R}dx.
\end{eqnarray}
Applying Holder inequality, the 1st term on the RHS of \eqref{no-p11} can be simplified as follows
\begin{eqnarray}\label{no-pp11}
\int_{\Omega\cap\{\epsilon\leq|x|\leq2\epsilon\}}|\Delta u||D^2 u|\frac{|x|}{\epsilon}dx &\leq& 2 \displaystyle\left(\int_{\Omega\cap\{\epsilon\leq|x|\leq2\epsilon\}}|\Delta u|^2dx\right)^\frac{1}{2}\left(\int_{\Omega\cap\{\epsilon\leq|x|\leq2\epsilon\}}|D^2 u|^2dx\right)^\frac{1}{2}\no\\
&\leq& 2\displaystyle\left(\int_{\Omega\cap\{\epsilon\leq|x|\leq2\epsilon\}}|\Delta u|^2dx\right)^\frac{1}{2}\left(\Iom|D^2 u|^2dx\right)^\frac{1}{2}\no\\
&\leq&C\displaystyle\left(\int_{\Omega\cap\{\epsilon\leq|x|\leq2\epsilon\}}|\Delta u|^2dx\right)^\frac{1}{2}.
\end{eqnarray}
Similarly, the 2nd term on the RHS of \eqref{no-p11}:
\begin{equation}\label{no-ppp11}
\int_{\Omega\cap\{R\leq|x|\leq2R\}}|\Delta u||D^2 u|\frac{|x|}{R}dx\leq C\displaystyle\left(\int_{\Omega\cap\{R\leq|x|\leq2R\}}|\Delta u|^2dx\right)^\frac{1}{2}.\end{equation}
Combining \eqref{no-p11},  \eqref{no-pp11} and  \eqref{no-ppp11} we find 
\begin{equation}\label{no-11}
\lim_{R\to\infty}\lim_{\epsilon\to 0}I_4=0.
\end{equation}
Similarly, it can be shown that
\begin{equation}\label{no-12}
\lim_{R\to\infty}\lim_{\epsilon\to 0}I_5=0.
\end{equation}
Therefore combining \eqref{no-lhs}, \eqref{no-8},\eqref{no-9},\eqref{no-10},\eqref{no-11} and \eqref{no-12} we have
\begin{equation}\label{no-13}
\lim_{R\to\infty}\lim_{\epsilon\to 0}\Iom(\Delta^2 u)(x\cdot\na u)\phi_{\epsilon,R}dx=-\displaystyle\left(\frac{N-4}{2}\right)\Iom|\Delta u|^2dx-\frac{1}{2}\int_{\bdw}|\Delta u|^2(x\cdot\nu)dS.
\end{equation}
Finally, taking into account \eqref{non-1}, \eqref{no-I_1}, \eqref{no-I_2} and \eqref{no-13} we get
\begin{equation*}
-\displaystyle\left(\frac{N-4}{2}\right)\Iom|\Delta u|^2dx-\frac{1}{2}\int_{\bdw}|\Delta u|^2(x\cdot\nu)dS=\displaystyle\left(\frac{N-4}{2}\right)\left[-\mu\Iom\frac{|u|^2}{|x|^4}dx-\Iom\frac{|u|^\qb}{|x|^{\ba}}dx\right].
\end{equation*}
Comparing this with the equation \eqref{eq:problem2} we find
$$\int_{\bdw}|\Delta u|^2(x\cdot\nu)dS=0.$$ 
Since $\Omega$ is star shaped, $x\cdot\nu>0$ on $\bdw$. This implies $\Delta u=0$ on $\bdw$ a.e. \\
Note that $u\in C^4(\bar\Omega\setminus B_r(0))$ for some $r>0$  [see \cite{GGS}]. Consequently, $\Delta u=0$ in $\bdw\setminus B_r(0)$. We choose a smooth subdomain $D\subset\Omega$ such that $\bar D\cap\bar B_r(0)=\emptyset$ and $\pa D\cap\bdw$ contains a $(N-1)$ dimensional open subset. Thus $u\in C^4(\bar D)$. By defining $v=-\Delta u$, we get  $-\Delta v\geq 0$ in $D$. Since $v=0$ on $\pa D\cap\bdw$, we have $v>0$ in $D$ that is, $-\Delta u>0$ in $D$. As $u=0$ on $\pa D\cap\bdw$, by applying Hopf's lemma for the strictly superharmonic function we obtain $\frac{\pa u}{\pa\nu}(x)<0$ where $x$ belongs to $(N-1)$ dimensional open subset of $\pa D\cap\bdw$, which contradicts that $\na u=0$ on $\bdw$.
\hfill{$\square$}
\end{proof}

\section{Palais-Smale characterization}
In this section we study the Palais-Smale sequences (PS sequences, in short)  of the functional \eqref{I-mu}, where $\Omega$ is a bounded domain with smooth boundary and $0\leq\ba<4$ is fixed. 
\begin{Definition}
A sequence $\{u_n\} \subset\Huno$ is called a PS sequence for $I_{\mu}$ at level $d$ if $I_{\mu}(u_n)\to d$ and $I_{\mu}'(u_n)\to 0$ in $H^{-2}(\Omega)$.
\end{Definition}
{\bf Remark:} By standard arguments it is not difficult to check that if $u_n$ is a PS sequence for $I_{\mu}$ at a level $d$, then $d\geq 0$.

\vspace{2mm}

It is easy to see that the weak limit of a PS sequence is a solution to \eqref{eq:problem}, except for the positivity. 
However the main difficulty is that the PS sequence may not converge strongly and hence the weak limit can be zero even if $d> 0$. In this section our aim is to classify PS sequences for $I_{\mu}$. Classification of PS sequences has been done for various problems with the second order operator for having lack of compactness, we refer to \cite{BS}, \cite{Smets},  \cite{Struwe1} among others. While the noncompactness studied in \cite{Smets} is due to a concentration phenomenon occurring through a double profile, in \cite{Struwe1} the noncompactness is a result of concentration occurring through single profiles. We derive a classification theorem for the PS sequences of \eqref{I-mu} in the same spirit of the above results. Here concentration takes place through one or two profiles depending whether $\ba>0$ or $\ba= 0$, respectively. This phenomenon was also observed in \cite{BS}.\\
Let $V$ be a solution to 
\begin{equation}\label{V}
\begin{cases}
\Delta^2u-\mu\frac{u}{|x|^4}=\frac{|u|^{q_{\ba}-2}u}{|x|^{\beta}} \quad\text{in}\quad \Rn,\\
u\in D^{2,2}(\Rn),
\end{cases}
\end{equation} see  \cite[Theorem 1.2]{BM}.
We define a sequence $v_n$ as follows :
\begin{equation}\label{ps-type1}
v_n(x) =\phi(\bar\la_n x)[\la_n^\frac{N-4}{2}V(\la_n x)],
\end{equation}
where $\phi\in C_0^{\infty}(B(0,2))$ with $\phi=1$ in $B(0,1)$ and $\la_n\to\infty$, $\frac{\bar\la_n}{\la_n}\to 0$.

 Then a simple computation shows that $\{v_n\}\subset\Huno$ is a PS sequence for $I_{\mu}$ at a level $d=I_{\mu}(V)$
where $I_{\mu}(V)$ is defined as in \eqref{I-mu} with $\Omega=\Rn$.\\
Assume $\mu>0$ and $\ba=0$. We define a sequence $w_n(x)$ as
\begin{equation}\label{ps-type2}
w_n(x)=\phi(\bar\la_n x)[\la_n^\frac{N-4}{2}W(\la_n(x-x_n)].
\end{equation}
where $W\in D^{2,2}(\Rn)$ satisfying $\Delta^2 W = |W|^{\starstar-2}W$, $\phi$ is as above, $x_n\in\Omega$, $|x_n|\la_n\to\infty$, $\frac{\bar\la_n}{\la_n}\to 0$ and $\liminf_{n\to\infty}\bar\la_n\text{dist}(x_n,\bdw)>2$ (see \cite{BS}, \cite{Smets} for similar type of computation in the case 2nd order version of equation \eqref{eq:problem}).
 Then it is not difficult to see that $w_n\in\Huno$, $w_n\deb 0$ and is also a PS sequence for $I_{\mu}$ at
level $d= I_0(W)$, where $I_0(W)$ is defined as in \eqref{I-mu} with $\Omega=\Rn$ , $\ba=0$ and $\mu=0$.
In fact, in the next theorem we prove that any noncompact PS sequence is  essentially a finite sum of sequences
of the form \eqref{ps-type1} and \eqref{ps-type2} when $\ba=0$ and $\mu>0$ and a finite sum of sequences of the form \eqref{ps-type1}
when $\ba>0$ .

\begin{Theorem}\label{ps-theorem} 
Let $\Omega$ be a bounded domain with smooth boundary and $0\in\Omega$. Let $\{u_n\}$ be a PS sequence for $I_{\mu}$ at level $d$. Suppose $\ba=0$ and $\mu>0$, then $\exists\ \  n_1, n_2\in\N$, and functions $v_n^j \in H^2_0(\Omega), 1 \leq j \leq n_1 $ and $w_n^k \in H^2_0(\Omega), 1 \leq k \leq n_2 $ and $u_0 \in  H^2_0(\Omega)  $ such that upto a subsequence \\
(1) $u_n= u_0+\sum\limits_{j=1}^{n_1}v_n^j +\sum_{k=1}^{n_2}w_n^k +o(1)$, where $o(1) \to 0$ in $H^2_0(\Omega)$\\
(2) $d = I_{\mu}(u_0) +\sum_{j=1}^{n_1}I_{\mu} (V^j) +\sum_{k=1}^{n_2}I_{0} (W^k) + o(1)$\\
where $I_{\mu}'(u_0)=0 $ and $v_n^j, w_n^k$  are PS sequences of the form \eqref{ps-type1} and \eqref{ps-type2} respectively with $V = V^j$ and $W=W^k$.\\
When $\ba>0$,  the same conclusion holds with $W^k =0$ for all $k.$
\end{Theorem}

{\bf Remark:} The case $\mu=0=\ba$ has been studied in \cite{GGS}, where the same conclusion of Theorem \ref{ps-theorem} holds with $V^k=0$ for all $k$.

We prove a lemma first.
\begin{Lemma}\label{l:ps}
Let $u_n$ be a PS sequence for $ I_{\mu}$ at a level $d<\frac{4-\ba}{2(N-\ba)}\left(S_{\mu}^{\ba}\right)^\frac{N-\ba}{4-\ba}$, where $S_{\mu}^{\ba}$ is defined as in \eqref{S-mu}.  Then $u_n$ is relatively compact in $\Huno$.\\
\end{Lemma}
\begin{proof}
By standard arguments it can be be shown that $u_n$ is bounded  $\Huno$. Therefore up to a subsequence 
$u_n\deb u$ in $\Huno,\; u_n\to u$ in $L^p(\Ome)$ for $p<\starstar$ and pointwise. Thus using Vitaly's convergence theorem we can show that $ I_{\mu}'(u) =0$ and hence
\begin{displaymath}
I_{\mu}(u)=\displaystyle\left(\frac{1}{2}-\frac{1}{\qb}\right)\Iom \frac{|u|^{\qb}}{|x|^{\ba}}dx \geq 0.
\end{displaymath}
Also by Brezis-Lieb lemma \cite{BL} we have
\begin{displaymath}
\Iom \frac{|u_n|^{\qb}}{|x|^{\ba}}dx=\Iom \frac{|u|^{\qb}}{|x|^{\ba}}dx+\Iom \frac{|u_n-u|^{\qb}}{|x|^{\ba}}dx+o(1), 
\end{displaymath}
\begin{displaymath}
 \Iom\frac{u_n^2}{|x|^4}dx=\Iom\frac{u^2}{|x|^4}dx+\Iom\frac{|u_n-u|^2}{|x|^4}dx+o(1).
\end{displaymath}
Hence, \begin{displaymath}
I_{\mu}(u_n)=I_{\mu}(u)+I_{\mu}(u_n-u)+o(1).
\end{displaymath}
We define $v_n:=u_n-u$. Consequently,
\begin{displaymath}
I_{\mu}(v_n)=I_{\mu}(u_n)-I_{\mu}(u)+o(1)\leq I_{\mu}(u_n)\leq d<\frac{4-\ba}{2(N-\ba)}\left(S_{\mu}^{\ba}\right)^\frac{N-\ba}{4-\ba}.
\end{displaymath}
We also have
$$ o(1)= \big<I_{\mu}'(u_n)-I_{\mu}'(u),v_n\big> $$
$$ = \int_\Ome\left[|\Delta v_n|^2- \mu\frac{v_n^2}{|x|^4}\right] dx- \int_\Ome \left[(|u_n|^{\qb-2}u_n -|u|^{\qb-2}u)\frac{v_n}{|x|^{\ba}} \right]dx.$$ 
As $\displaystyle\int_\Ome \frac{|u_n|^{\qb-2}u_nu}{|x|^{\ba}} dx\to \int_\Ome \frac{|u|^{\qb}}{|x|^{\ba}}dx$ and $\displaystyle\int_\Ome \frac{|u|^{\qb-2}u u_n}{|x|^{\ba}} dx\to \int_\Ome \frac{|u|^{\qb}}{|x|^{\ba}}dx$, by using Brezis-Lieb Lemma \cite{BL}, we simplify the last integral as
$$ \int_\Ome \left[(|u_n|^{\qb}- |u|^{\qb})\frac{1}{|x|^{\ba}}\right]dx + o(1) = \int_\Ome\frac{|v_n|^{\qb}}{|x|^{\ba}} dx+ o(1). $$
Hence 
\begin{equation}\label{ps-vn}
 \int_\Ome|\Delta v_n|^2dx-\mu\int_\Ome\frac{v_n^2}{|x|^4}dx-\int_\Ome\frac{|v_n|^{\qb}}{|x|^{\ba}} dx=o(1).
\end{equation}
As a result,
$$ I_{\mu}(v_n)=\frac{4-\ba}{2(N-\ba)}\Iom \frac{|v_n|^{\qb}}{|x|^{\ba}}dx\leq
 d<\frac{4-\ba}{2(N-\ba)}\left(S_{\mu}^{\ba}\right)^\frac{N-{\ba}}{4-\ba}.$$
In consequence,
\begin{equation}\label{ps-l}
 \Iom\frac{|v_n|^{\qb}}{|x|^{\ba}} dx\leq \delta \Iom \left(|\Delta v_n|^2-\mu \frac{v_n^2}{|x|^4} \right)dx, \;\; {\rm where}\;\; 0< \delta <1.
\end{equation}
Substituting \eqref{ps-l} in \eqref{ps-vn} we get
$$ (1-\delta)\Iom \left(|\Delta v_n|^2-\mu\frac{v_n^2}{|x|^4} \right) dx= o(1) . $$  Hence $v_n\to 0$ in $\Huno$ and the lemma follows.
\hfill{$\square$} 
\end{proof}

\vspace{2mm}

{\bf Proof of Theorem \ref{ps-theorem}:}  We break this proof into two steps:

\vspace{2mm}

{\bf Step 1:} Let $u_n$ be a PS sequence for $I_{\mu}$, converging weakly to $0$. We show that, up to a subsequence either $u_n \to 0$ in $\Huno$ or there exists a PS sequence $\tilde u_n$ of $I_{\mu}$ such that $I_{\mu}(u_n) = I_{\mu}(\tilde u_n) +I_{\mu}(u_n-\tilde u_n) +o(1)$, $u_n-\tilde u_n$ is again a PS sequence for $I_{\mu}$ and
$\tilde u_n$ is of the form \eqref{ps-type1} or \eqref{ps-type2}. If $\ba>0$, then $\tilde u_n$ must be of the form \eqref{ps-type1}.\\

\vspace{2mm}

If there does not exist any subsequence of $u_n$ which converges to zero, then in view of Lemma \ref{l:ps}, we may assume that $\liminf\limits_{n \rightarrow \infty}I_{\mu}(u_n) \ge \frac{4-\ba}{2(N-\ba)}\left(S_{\mu}^\ba\right)^\frac{N-\ba}{4-\ba}$. 
Since $u_n\deb 0$ in $H^2_0(\Omega)$, an equality of the form \eqref{ps-vn} can be shown. Hence up to a subsequence
$$\lim\limits_{n \rightarrow \infty} \Iom \frac{|u_n|^{\qb}}{|x|^{\ba}}dx
 \ge \left(S_{\mu}^{\ba}\right)^\frac{N-\ba}{4-\ba}.$$
Let $ Q_n(r)$ denote the concentration function
\begin{displaymath}
 Q_n(r)=\int_{B_r(0)}\frac{|u_n|^{\qb}}{|x|^{\ba}}dx.
\end{displaymath}
Therefore we can choose $\la_n\geq \la_0>0$ such that 
$$Q_n(\la_n) = \int_{B_{\la_n}(0)}\frac{|u_n|^{\qb}}{|x|^{\ba}} dx= \delta ,$$
where $\delta $ is chosen such that $\delta^{\frac{4-\ba}{N-\ba}}<S_{\mu}^{\ba}$. Define 
$$v_n(x)={\la}_n^{-\frac{N-4}{2}}u_n(\frac{x}{\la_n}),\; x \in \Ome_n,$$ 
where $\Ome_n = \{x\in\Rn :\frac{x}{\la_n} \in \Ome\}$ and extend it to all of $\Rn$ by putting $0$ outside $\Ome_n$. Then $v_n \in D^{2,2}(\Rn)$ with $\text{supp}\ v_n \subset \Ome_n$  and satisfies
\begin{eqnarray} \label{F}
\int_{B_1(0)}\frac{|v_n|^{\qb}}{|x|^{\ba}}dx=\delta .
\end{eqnarray}
Since $\| v_n\|_{D^{2,2}(\Rn)}= \|u_n\|_{D^{2,2}(\Ome)}<\infty$, up to a subsequence we may assume $v_n\deb v_0$ in $D^{2,2}(\Rn)$. Now we consider two cases:\\

{\bf Case 1:} $v_0\not= 0$.\\
Here we note that since $\Ome$ is a bounded domain and $u_n\deb 0$ in $H^2_0(\Omega)$, we get the sequence $\la_n\to\infty$ and $\Omega_n\to\Rn$ as $n\to\infty$. For any $\phi\in C_0^{\infty}(\Rn)$, we define $\tilde\phi_n(x)=\la_n^\frac{N-4}{2}\phi(\la_n x)\in\Huno$. Note that $\phi\in C_0^{\infty}(\Omega_n)$ for large $n$. Then a straight forward calculation and the fact that $I'_{\mu}(u_n)\to 0$ in $H^{-2}(\Omega)$  yields to
\begin{equation}\label{new-1}
o(1) = (I_{\mu}'(u_n),\tilde\phi_n)=\displaystyle\int_{\Omega_n}\left[\Delta v_n\Delta\phi-\mu \frac{v_n\phi}{|x|^4}-\frac{|v_n|^{\qb-2}v_n\phi}{|x|^{\ba}}\right]dx.
\end{equation}
Thanks to \eqref{Rellich} and \eqref{CKN} if we take the limit $n\to\infty$, we have, 
\begin{equation}\label{new-1'}
\displaystyle\int_{\Rn}\left[\Delta v_0\Delta\phi-\mu \frac{v_0\phi}{|x|^4}-\frac{|v_0|^{\qb-2}v_0\phi}{|x|^{\ba}}\right]dx=0.
\end{equation}
 Thus $v_0$ solves \eqref{V}.

Let $ \va\in C_0^\infty(\Rn)$ such that $ 0\leq\va\leq 1,\ \ \va\equiv 1$ in $B_1(0),\ \ supp \ \va\subseteq B_2(0)$. Define
\begin{equation}\label{buble-1}
\tilde u_n(x)= {\la}_n^\frac{N-4}{2}v_0(\la_n x)\va(\bar\la_n x),
\end{equation}
where $\bar\la_n>0$ is chosen s.t $\tilde\la_n=\frac{\bar\la_n}{\la_n} \to 0$. Thus we have $\bar\la_n\mbox{dist}(0,\pa\Ome_n)\to\infty$ as $n\to\infty$. Clearly $\tilde u_n$ is a PS sequence of the form \eqref{ps-type1}. Next we prove the splitting of energy.  By applying  Brezis-Lieb Lemma we see that $I_{\mu}(u_n)$
can be written as
\begin{eqnarray*}
I_{\mu}(u_n)=I_{\mu}(v_n)&=&\frac{1}{2}\int\limits_{\Rn}\left[|\Delta v_n|^2- \mu\frac{v_n^2}{|x|^4}\right]dx -\frac{1}{\qb}\int_{\Rn}\frac{|v_n|^{\qb}}{|x|^{\ba}}dx\\
&=&\frac{1}{2}\int\limits_{\Rn}\left[|\Delta v_0|^2dx- \mu\frac{v_0^2}{|x|^4}\right]dx -\frac{1}{\qb}\int_{\Rn}\frac{|v_0|^{\qb}}{|x|^{\ba}}dx\\
&+& \frac12\int\limits_{\Rn}\left[|\Delta (v_n- v_0)|^2- \mu\frac{(v_n-v_0)^2}{|x|^4}\right]dx \\
&-&\frac{1}{\qb}\int\limits_{\Rn}\frac{|v_n-v_0|^{\qb}}{|x|^{\ba}} dx+o(1). 
\end{eqnarray*}
Next, we define $\va_n(x)=\va(\tilde\la_nx)$, then
\begin{eqnarray*}
\int_{\Rn}|\Delta(v_0\va_n-v_0)|^2 &\leq&  C\int_{\Rn}|\Delta v_0|^2(\va_n-1)^2 +
C\int_{\Rn}|v_0|^2|\Delta\va_n|^2 +C\int_{\Rn}|\na v_0|^2|\na\va_n|^2\\
& \leq& 
 C\int\limits_{|x|> \frac{1}{\tilde\la_n}}|\Delta v_0|^2 + C\big(\int\limits_{\frac{1}{\tilde\la_n} < |x| <\frac{2}{\tilde\la_n}}|v_0|^{\starstar}\big)^{2/\starstar}+C\int\limits_{\frac{1}{\tilde\la_n} < |x| < \frac{2}{\tilde\la_n}}|\na v_0|^2.
\end{eqnarray*}
As $v_0\in D^{2,2}(\Rn)$ implies that $\na  v_0\in D^{1,2}(\Rn)$, from the above relation we obtain  $v_0\va_n\to v_0$ in $D^{2,2}(\Rn)$. Hence 
\begin{eqnarray*}
I_{\mu}(u_n)&=& \frac{1}{2}\int\limits_{\Rn}\left[|\Delta (\va_nv_0)|^2- \mu\frac{(\va_nv_0)^2}{|x|^4}\right] dx-\frac{1}{\qb}\int\limits_{\Rn}\frac{|\va_nv_0|^{\qb}}{|x|^{\ba}}dx\\
&+& \frac{1}{2}\int\limits_{\Rn}\left[|\Delta (v_n- \va_nv_0)|^2- \mu\frac{(v_n-\va_nv_0)^2}{|x|^4}\right]dx -\frac{1}{\qb}\int\limits_{\Rn}\frac{|v_n-\va_nv_0|^{\qb}}{|x|^{\ba}} dx+o(1).
\end{eqnarray*}

Then by a change of variables we obtain
$I_{\mu}(u_n)=I_{\mu}(\tilde u_n)+I_{\mu}(u_n-\tilde u_n)+o(1) $.\\
Using similar type of arguments we can show $I_{\mu}^\prime(u_n-\tilde u_n)= o(1)$ in $H^{-2}(\Ome)$.

\vspace{2mm}

{\bf Case 2:} $v_0=0$\\

\vspace{2mm}

Let $\va\in C_0^\infty B_1(0)$ with $0\le \va \le 1$. By taking $\va^2v_n$ as a test function in \eqref{new-1}, we have
\begin{equation}\label{phivn}
\int_{\Rn}\left(\Delta v_n\Delta(\va^2v_n)-\mu \frac{(\va v_n)^2}{|x|^4} \right)dx = \int_{\Rn} \frac{|v_n|^{\qb-2}(\va v_n)^2}{|x|^{\ba}}dx + o(1).
\end{equation}
By using Rellich's compactness theorem, we have  $v_n\to 0$ in $L^2_{loc}(\Rn)$ and $D^{1,2}_{loc}(\Rn)$. Therefore, by doing a simple computation the LHS of \eqref{phivn} can be simplified as \\
$\displaystyle\int_{\Rn}\left(|\Delta(\va v_n)|^2-\mu \frac{(\va v_n)^2}{|x|^4} \right) dx+ o(1)$, whereas the RHS can be simplified as
$$\int_{\Rn} \frac{|v_n|^{\qb-2}(\va v_n)^2}{|x|^{\ba}} dx\leq \left(\int_{B_1(0)} \frac{|v_n|^{\qb}}{|x|^{\ba}}dx\right)^{\frac{\qb-2}{\qb}}\left(\int_{\Rn}\frac{|\va v_n|^{\qb}}{|x|^{\ba}}dx\right)^{\frac{2}{\qb}}$$
$$\leq \frac{\delta^{\frac{4-\ba}{N-\ba}}}{S_{\mu}^{\ba}}\int_{\Rn}\left(|\Delta(\va v_n)|^2-\mu \frac{(\va v_n)^2}{|x|^4} \right) dx.$$
Note that by the choice of $\delta$ we had $\delta^\frac{4-\ba}{N-\ba}<S_{\mu}^{\ba}$, as a consequence, from \eqref{phivn} we have 
\begin{equation}\label{G}
\int_{\Rn}\left(|\Delta(\va v_n)|^2-\mu \frac{(\va v_n)^2}{|x|^4} \right) dx= o(1), 
\end{equation}
 which in turn implies  
\begin{equation*} 
 \displaystyle \int_{\Rn}\frac{|\va v_n|^{\qb}}{|x|^{\ba}}dx = o(1) 
 \end{equation*} 
(since a formula similar to \eqref{ps-vn} holds for $\va v_n$ as $ \va v_n\deb0$) .
 Hence 
\begin{equation*}
 \displaystyle\int_{B_r(0)}\frac{| v_n|^{\qb}}{|x|^\ba} dx= o(1)\quad\forall\quad 0<r<1.
 \end{equation*} 
If $\ba>0$, this implies 
\begin{equation}
\displaystyle\int_{K}\frac{| v_n|^{\qb}}{|x|^{\ba}} dx= o(1)
\end{equation}
for any compact set $K \subset \Rn$, which contradicts \eqref{F}. Therefore when $\ba>0$, the Case 2 cannot happen and we are through. Thus we assume now onwards $\ba=0.$ \\\\
The condition \eqref{F} together with the concentration compactness principle \cite{L} (see also \cite{BS}, \cite{GGS}, \cite{Smets}) give
\begin{displaymath}
 |v_n|^{\starstar}dx\big|_{\{|x|\leq 1\}}\deb\sum_j C_{x_j}\delta_{x_j},
\end{displaymath}
where the points $x_j\in\Rn$ satisfy $|x_j|=1$. Let $C=max\{C_{x_j}\}$ and define 
\begin{displaymath}
 \tilde{Q}_n(r)=\sup_{x\in\Rn}\int_{B_r(x)}{|v_n|}^{\starstar}dy.
\end{displaymath}
Clearly, $\tilde{Q}_n(\infty)>\frac{C}{2}$ and for each $r>0$ we have
$ \liminf_{n\to\infty}\tilde{Q}_n(r)>\frac{C}{2}.$
Hence there exists  a sequence $\{q_n\} \subset\Rn$ and $\{s_n\}\subset\R^{+}$ s.t $s_n\to 0$ and  $|q_n|>\frac{1}{2}$ and
\begin{equation}\label{con-2}
 \frac{C}{2}=\sup_{q\in\Rn}\int_{B_{s_n}(q)}| v_n|^{\starstar}dx=\int_{B_{s_n}(q_n)}|v_n|^{\starstar}dx.
\end{equation}
We define $z_n(x)=s_n^\frac{N-4}{2}v_n(s_nx+q_n)=(\frac{s_n}{\la_n})^\frac{N-4}{2}u_n(\frac{s_nx+q_n}{\la_n})$.\\
Therefore upto a subsequence we can assume that $\exists z\;\in D^{2,2}(\Rn)$ s.t $z_n\deb z$ in $D^{2,2}(\Rn)$ and $z_n(x)\to z(x)$ a.e.\\ We claim that $z \not= 0$, otherwise we choose $\va\in C_0^\infty(B_1(x))$ where $0\le \va \le 1$ and $x \in \R^N$  is chosen arbitrarily but fixed. Then proceeding the same way as we established (\ref{G}), we can show that $z_n\to 0$ in $L_{loc}^{\starstar}(\Rn)$ which contradicts \eqref{con-2}. \\
we also observe that 
$$ z_n(x)=(\tilde\la_n)^\frac{N-4}{2}u_n(\tilde\la_n x+\tilde x_n), $$ where $\tilde\la_n=\frac{s_n}{\la_n}=o(1)$ and $\tilde x_n=\frac{q_n}{\la_n}$. We have $\frac{\tilde\la_n}{|\tilde x_n|}<2s_n=o(1)$.\\\\
Support of $z_n \subset \tilde\Ome_n := \{x \in\Rn: \tilde\la_n x+ \tilde x_n \in \Ome \}$ and $\tilde\Ome_n$ exhausts as $\tilde\Ome_\infty$ which is either a half space or $\Rn$ depending whether $\lim \tilde\la_n {\text dist}(\tilde x_n, \partial \tilde\Ome_n)$ is finite or infinite.
Since $\{u_n\}$ is a PS sequence,  we have
\begin{displaymath}
 \Iom\Delta u_n\Delta\va dx-\mu\Iom\frac{u_n\va}{|x|^4}dx-\Iom |u_n|^{\starstar-2}u_n\va dx=o(\|\va\|) \ \
 \mbox{$\forall\ \va\in\Huno$}.
\end{displaymath}
Let $\psi \in C_0^\infty(\tilde\Ome_\infty)$.We choose $\va=\va_n $ as $\va_n(x)= \tilde\la_n^{-\frac{N-4}{2}}\psi\displaystyle\left(\frac{x-\tilde{x}_n}{\tilde\la_n}\right)$ in the above relation, then a change of variables together with the fact that $ \|\va_n\| = \|\psi\|$ leads to
$$
\int_{\Omega_n}\Delta z_n\Delta\psi dx-\mu\int_{\Omega_n}\frac{z_n\psi}{|x+\frac{\tilde  x_n}{\tilde\la_n}|^4}dx
=\int_{\Omega_n}|z_n|^{\starstar-2}z_n\psi dx+o(\|\psi\|).
$$
As a result, taking the limit as $n \to \infty$ and using 
 $\frac{\tilde{x}_n}{\tilde{\la_n}}\to\infty$ we find
$$
\int_{\Omega_{\infty}}\Delta z\Delta\psi dx=\int_{\Omega_{\infty}}|z|^{\starstar-2}z\psi dx.$$
However by the well known Pohazaev nonexistence result, this is not possible when $\tilde\Ome_\infty$ is a half space (see \cite{GGS}). Therefore we get $\lim \tilde\la_n {\text dist}(\tilde x_n, \partial \tilde\Ome_n) = \infty$
Now we define 
\begin{equation}\label{buble-2}
\tilde u_n(x)= \tilde\la_n^{-\frac{N-4}{2}}z\displaystyle\left(\frac{x-\tilde{x}_n}{\tilde \la_n}\right)\va\big(\bar\la_n(x-\tilde{x}_n)\big) ,
\end{equation}
where $\va$ is as in \eqref{buble-1}, $\bar\la_n$ is chosen s.t $\tilde \la_n\bar\la_n \to 0$ and $\bar\la_n\mbox{dist}(\tilde{x}_n,\pa\Ome)\to\infty$ as $n\to\infty$. Proceeding exactly as in the case of \eqref{buble-1}, we see that 
$\tilde u_n, u_n-\tilde u_n$ are PS sequences and $I_{\mu}(u_n) = I_{\mu}(\tilde u_n) + I_{\mu}(u_n-\tilde u_n) +o(1).$ 

\vspace{2mm}

{\bf Step 2:} In this step we complete the proof of the theorem. If  $d<\frac{4-\ba}{2(N-\ba)}\left(S_{\mu}^\ba\right)^\frac{N-\ba}{4-\ba}$, then we are done. Otherwise, since $\{u_n\}$ is a PS sequence at level $d$, $\{u_n\}$ is bounded in $\Huno$ and hence we may assume $u_n$ converges weakly to $u \in \Huno$. Using standard arguments it is not difficult to see that $u_n-u$ is a PS sequence converging weakly to zero and $d = I_{\mu}(u) + I_{\mu}(u_n-u) +o(1)$. Now either $u_n-u$ is a PS sequence at a level $\tilde d<\frac{4-\ba}{2(N-\ba)}\left(S_{\mu}^\ba\right)^\frac{N-\ba}{4-\ba}$, or we can find a $\tilde u_n $ as in Step 1. We observe that $ I_{\mu}(\tilde u_n) $ converges either to $I_{\mu}(V)$ or $I_0(W)$ where $V$ and $W$ are as in \eqref{ps-type1} and \eqref{ps-type2} 
and $I_{\mu}(V), I_0(W) \geq C_1 >0$. Therefore in finitely many steps we get a PS sequence at a level strictly less than $\frac{4-\ba}{2(N-\ba)}\left(S_{\mu}^\ba\right)^\frac{N-\ba}{4-\ba}$ and this completes the proof.
\hfill$\square$

\section{Existence result}
In this section we prove the existence result in bounded domains which have non trivial topology in the spirit of \cite{HS} when $\ba=0$.  More precisely, here we consider the problem:
\begin{equation}
\label{eq:problem1}
\begin{cases}
  \Delta^2u-\mu\frac{u}{|x|^4}=|u|^{2^{**}-2}u &\textrm{in $\Omega$},\\
u\in  H^2_0(\Omega),\\
u>0 \quad\text{in}\ \Omega.
\end{cases}
\end{equation}
For the 2nd order equation: $-\Delta u=|u|^{\frac{4}{N-2}}u, \quad u\in H^1_0(\Omega)$, existence of positive solutions in domains with non trivial topology  was first studied by Coron \cite{Cor}.  In \cite{HS}, existence result was studied for the 2nd order version of the problem \eqref{eq:problem1}.

Now we state the main result of this section.

\begin{Theorem}\label{t:existence}
Let $\Omega$ be a bounded domain which is not contractible, $0\in\Omega$ and $N\geq 8$. Then there exists $\mu_0\in(0,\bar\mu)$ such that $\forall\ \mu\in(0,\mu_0)$ there exists a solution to \eqref{eq:problem1}.
\end{Theorem}

In order to prove this theorem, we need to introduce some notations and recall some previous results.

\vspace{2mm}

{\bf Notation:} If $A$ and $B$ are two subsets of $\Rn$, by $A\cong B$ we mean $A$ and $B$ are homotopy equivalent. $\Omega_d=\{x\in\Rn: dist(x,\Omega)<d\}$ and $\Omega_d^i=\{x\in\Omega: dist(x,\partial\Omega)>d\}$. When $\ba=0$, we simply denote $S_{\mu}^{\ba}$ as $S_{\mu}$.

\vspace{2mm}

From \cite{BM} it is been known that \eqref{eq:problem1} has a unique radial solution say,  $u^{\mu}_0$ when $\Omega=\Rn$. Furthermore, in \cite{BM} it is shown that any ground state solution to equation \eqref{eq:problem1} is radially symmetric. Therefore $u_0^{\mu}$ is only possible ground state solution. Let $u_0^0$ denote the unique solution of \eqref{eq:problem1} in $\Rn$ when $\mu=0$. For $\epsilon>0$ and $z\in\Rn$,  we define
$$u^{\mu}_{z,\epsilon}=\epsilon^{-(\frac{N-4}{2})}u^{\mu}_0\displaystyle\left(\frac{x-z}{\epsilon}\right).$$
Then for each $\epsilon>0$, $u^{\mu}_{0,\epsilon}$ is a solution to \eqref{eq:problem1} in $\Rn$ with the energy level $I_{\mu}(u^{\mu}_{0,\epsilon})=\frac{2}{N}S_{\mu}^{\frac{N}{4}}$ and $u^{0}_{z,\epsilon}$ is a solution to \eqref{eq:problem1} in $\Rn$ with $\mu=0$ and with the same energy level $I_{0}(u^{0}_{z,\epsilon})=\frac{2}{N}S_{0}^{\frac{N}{4}} \quad\forall\quad z\in\Rn \quad\text{and}\quad\quad\forall \epsilon>0$. 
 
Clearly by definition of $S_{\mu}$, we have $\frac{2}{N}S_{\mu}^{\frac{N}{4}}<\frac{2}{N}S_{0}^{\frac{N}{4}}$ when $\mu>0$. We define the Nehari Manifold $N_{\mu}$ as follows:
$$N_{\mu}=\displaystyle\left\{u\in D^{2,2}(\Rn)\setminus\{0\}: \int_{\Rn}\displaystyle\left[|\Delta u|^2-\mu\frac{u^2}{|x|^4}\right]dx=\int_{\Rn}|u^{\starstar}|dx\right\}.$$
We set $N_{\mu}(\Omega)=N_{\mu}\cap H^2_0(\Omega)$. Now we define,
$$\ga(u)=\displaystyle\left[\frac{\displaystyle\int_{\Rn}\displaystyle\left(|\Delta u|^2-\mu\frac{u^2}{|x|^4}\right)dx}{\displaystyle\int_{\Rn}|u^{\starstar}|dx}\right]^\frac{N-4}{8}, \quad\forall\quad u\in D^{2,2}(\Rn)\setminus \{0\}.$$ Then a simple calculation shows that $\quad\forall\ u\in D^{2,2}(\Rn)\setminus\{0\}$ we have $\ga(u).u\in N_{\mu}$.

Note that since $S_{\mu}\longrightarrow S_0$ as $\mu\longrightarrow 0$, we can choose $\mu'\in(0,\bar\mu)$ such that
\begin{equation}\label{mu-dash}
2^\frac{4}{N}S_{\mu}>S_0 \quad\forall\quad \mu\in(0,\mu').
\end{equation}

\vspace{2mm}

We break the proof of Theorem \ref{t:existence} into several lemmas and propositions.

\begin{Proposition}\label{p:PS}
Assume that there does not exist any solution to problem \eqref{eq:problem1}.

\vspace{2mm}

(i) Set $\mu\in(0,\mu')$ as in \eqref{mu-dash}. If $\{u_n\}\subset N_{\mu}(\Omega)$ is a PS sequence  of $I_{\mu}$ at level $d<\frac{2}{N}S_0^\frac{N}{4}$, then there exists a sequence $\{\epsilon_n\}\in\R^{+}$ such that
$ \epsilon_n\to 0$ and $||u_n-u^{\mu}_{0,\epsilon_n}||\to 0$. 

(ii)Set $\mu=0$. If $\{u_n\}\in\Huno$ verifies
$$I_0(u_n)\leq\frac{2}{N}S_0^\frac{N}{4} \quad\text{and}\quad \lim_{n\to\infty}|\Delta u_n|^2_{L^2(\Omega)}=\lim_{n\to\infty}|u_n|^{\starstar}_{L^{\starstar}(\Omega)},$$
then there exists a sequence $\{z_n\}\subset\Rn$ and $\{\epsilon_n\}\subset\R^{+}$ such that $\epsilon_n\to 0$ and \\ $||u_n-u^{0}_{z_n,\epsilon_n}||\to 0$. 
\end{Proposition}

\begin{proof}
(i) From  Theorem \ref{ps-theorem} we know that 
$$u_n=u_0+\sum_{j=1}^{n_1}v_n^j+\sum_{k=1}^{n_2}w_n^k+o(1), $$ where $v_n^j$ and $w_n^k$ are PS sequences of the form \eqref{ps-type1} and \eqref{ps-type2}.
By the assumption  \eqref{eq:problem1} does not have any solution and $I_{\mu}(u_n)<\frac{2}{N}S_0^\frac{N}{4}$. Therefore $u_n$ must be of the form $\sum_{j=1}^{n_1}v_n^j+o(1)$. Note that by the choice of $\mu$ we have $2^\frac{4}{N}S_{\mu}>S_0$, thus if $n_1>1$, then
$$I_{\mu}(u_n)\geq\frac{4}{N}S_{\mu}^\frac{N}{4}>\frac{2}{N}S_{0}^\frac{N}{4},$$ which is a contradiction to the assumption and hence $n_1=1$. Therefore $u_n=u^{\mu}_{0,\epsilon_n}+o(1)$ where $\epsilon_n\to 0$ as $n\to\infty$.

(ii)This proof runs with similar arguments to part (i).
\hfill{$\square$}
\end{proof}

For $u\in\Huno\setminus\{0\}$, define the centre of mass as 
$$F(u)=\frac{\displaystyle\Iom x|\Delta u|^2 dx}{\displaystyle\Iom|\Delta u|^2dx}.$$ We choose $d_1, d_2>0$ such that $\Omega_{d_1}\cong\Omega\cong\Omega^i_{d_2}$.

\begin{Lemma}\label{center of mass}
There exists a $\mu_0\in(0,\mu')$ such that if $\mu\in(0,\mu_0)$ and  $u\in N_{\mu}(\Omega)$ with $I_{\mu}(u)<\frac{2}{N}S_0^\frac{N}{4}$, then $F(u)\in\Omega_{d_1}$.
\end{Lemma}

\begin{proof}
We prove this by method of contradiction. We assume that there exists a sequence $\{\mu_n\}\subset(0,\mu')$ and $u_n\in N_{\mu}(\Omega)$ such that $\mu_n\to 0$ and 
$I_{\mu}(u_n)<\frac{2}{N}S_0^\frac{N}{4}$ but $F(u_n)\not\in\Omega_{d_1}$. Without loss of generality we may assume that $F(u_n)\to z_0\not\in\Omega_{d_1}$. As $\mu_n\to 0$, using Rellich's inequality we obtain $$\lim_{n\to\infty}\mu_n\displaystyle\Iom\frac{|u_n|^2}{|x|^4}dx=0.$$ This implies,
$$I_0(u_n)\leq\frac{2}{N}S_0^\frac{N}{4} \quad\text{and}\quad \lim_{n\to\infty}|\Delta u_n|^2_{L^2(\Omega)}=\lim_{n\to\infty}|u_n|^{\starstar}_{L^{\starstar}(\Omega)}.$$
Consequently, by Proposition \ref{p:PS}(ii) we find a sequence $\{z_n\}\subset\Rn$ and $\{\epsilon_n\}\subset\R^{+}$ such that $\epsilon_n\to 0$ and $\lim_{n\to\infty}||u_n-u^0_{z_n,\epsilon_n}||=0$. Hence it implies $z_n\to z_0$. On the other hand as $z_0\not\in\Omega_{d_1}$, we get $\lim_{n\to\infty}||u_n-u^0_{z_n,\epsilon_n}||>0$ which is a contradiction and thus the lemma follows.
\hfill{$\square$}
\end{proof}

We define \begin{equation*}
\la:=\inf\displaystyle\left\{\frac{\mu}{|x|^4}:x\in\Omega_{d_1}\right\}>0.
\end{equation*} 
Take $\phi\in C_0^{\infty}(\Rn)$ such that $\phi=1$ in $\{|x|<\frac{d_2}{2}\}$ and $\phi=0$ in $\{|x|\geq d_2\}$. Define
$$v_{z,\epsilon}(x):=\phi(x-z)u^{0}_{z,\epsilon}(x).$$

\begin{Lemma}\label{l:estimates_v}
Let $\mu\in(0,\mu_0)$ and $N\geq 8$. Then there exists  $\epsilon>0$ such that 
$$\sup\big\{I_{\mu}\big(\ga(v_{z,\epsilon})v_{z,\epsilon}\big):  z\in\Omega^i_{d_2}\big\}<\frac{2}{N}S_0^\frac{N}{4}.$$
\end{Lemma}

\begin{proof}
Let $v_0:=\displaystyle\left(\frac{1}{1+|x|^2}\right)^\frac{N-4}{2}$. It's well known that $v_0$ is the unique solution to \eqref{eq:problem1} in $\Omega=\Rn$ with $\mu=0$ (see \cite{S}). Let us recall here some results from \cite{BAP}.\\
\begin{itemize}
\item[(i)]$\displaystyle\int_{\Omega}|\Delta v_{z,\epsilon}|^2 dx=\int_{\Rn}|\Delta v_0|^2dx+O(\epsilon^{N-4})dx=S_0^\frac{N}{4}+O(\epsilon^{N-4})$.
\item[(ii)]$\displaystyle\int_{\Omega}|v_{z,\epsilon}|^{\starstar}dx=\int_{\Rn}|v_0|^{\starstar}dx+O(\epsilon^{N})dx=S_0^\frac{N}{4}+O(\epsilon^{N})$.\\
\item[(iii)]\[\int_{\Omega}|v_{z,\epsilon}|^2dx=\left\{\begin{array}{lll}
k\epsilon^4+o(\epsilon^4) \quad \text{if}\ N>8,\\
k\epsilon^4|log\epsilon|+o(\epsilon^4|log\epsilon|) \quad \text{if}\ N=8,\\ 
k\epsilon^{N-4}+o(\epsilon^{N-4}) \quad \text{if}\ N<8.
\end{array}
\right.\]
\end{itemize}
Note that in (iii) $k$ is a positive constant.

\vspace{2mm}

Using the above estimates and the definition of $\la$ i.e.,  $\la\leq\frac{\mu}{|x|^4} \quad\forall\ x\in\Omega$, we estimate $I_{\mu}(tv_{z,\epsilon})$ for any $t>0$:

\begin{eqnarray*}
I_{\mu}(tv_{z,\epsilon})&=&\frac{t^2}{2}\displaystyle\int_{\Omega}\left[|\Delta v_{z,\epsilon}|^2-\mu\frac{|v_{z,\epsilon}|^2}{|x|^4}\right]dx-\frac{t^{\starstar}}{\starstar}\Iom|v_{z,\epsilon}|^{\starstar}dx\\
&\leq&\frac{t^2}{2}\displaystyle\int_{\Omega}\left[|\Delta v_{z,\epsilon}|^2-\la|v_{z,\epsilon}|^2\right]dx-\frac{t^{\starstar}}{\starstar}\Iom|v_{z,\epsilon}|^{\starstar}dx\\
&\leq&\displaystyle\left(\frac{t^2}{2}-\frac{t^{\starstar}}{\starstar}\right)S_0^\frac{N}{4}-\frac{t^2\la}{2}\Iom|v_{z,\epsilon}|^2dx+O(\epsilon^{N-4})+O(\epsilon^N).
\end{eqnarray*}
Here we note that $\sup_{t>0}(\frac{t^2}{2}-\frac{t^{\starstar}}{\starstar})S_0^\frac{N}{4}\leq(\frac{1}{2}-\frac{1}{\starstar})S_0^\frac{N}{4}=\frac{2}{N}S_0^\frac{N}{4}$.
Therefore
\[I_{\mu}(tv_{z,\epsilon})\leq \left\{\begin{array}{lll}\frac{2}{N}S_0^\frac{N}{4}-\frac{t^2\la}{2}k\epsilon^4+o(\epsilon^4)+O(\epsilon^{N-4})+O(\epsilon^N) \quad\text{if}\ N>8,\\
 \frac{2}{N}S_0^\frac{N}{4}-\frac{t^2\la}{2}k\epsilon^4|log\epsilon|+o(\epsilon^4|log\epsilon|)+O(\epsilon^{4})+O(\epsilon^8) \quad\text{if}\ N=8,\\
\frac{2}{N}S_0^\frac{N}{4}-\frac{t^2\la}{2}k\epsilon^{N-4}+o(\epsilon^{N-4})+O(\epsilon^{N-4})+O(\epsilon^N) \quad\text{if}\ N<8.
\end{array}
\right.\]
As a result if  $N\geq 8$,  we can choose $\epsilon>0$ small enough such that $I_{\mu}(tv_{z,\epsilon})<\frac{2}{N}S_0^\frac{N}{4}$ for every $z\in\Omega_{d_2}^i$ and $t>0$. Hence by taking $t=\ga(v_{z,\epsilon})$ we obtain the result.
\hfill{$\square$}
\end{proof}

\vspace{2mm}

{\bf Proof of Theorem  \ref{t:existence}:}\ Suppose equation \eqref{eq:problem1} does not have any solution. Let $\phi:[0,\infty)\times N_{\mu}\to N_{\mu}$ be a pseudo gradient flow associated with 
$I_{\mu}$ (see \cite{Struwe} and \cite{HS}). Namely, $\phi$ satisfies:

\begin{itemize}
\item[(i)]For $t,s\in\R^{+},\ t>s$ and $u\in N_{\mu}$ with $I_{\mu}'(u)\not=0$,
$$I_{\mu}(\phi(t,u))<I_{\mu}(\phi(s,u)).$$
\item[(ii)]$\lim_{t\to\infty}I_{\mu}(\phi(t,u))>-\infty \Longrightarrow
\lim_{t\to\infty}I'_{\mu}(\phi(t,u))=0.$
\end{itemize}
Define \begin{equation}\label{set-V}
 V=\{\ga(v_{z,\epsilon})v_{z,\epsilon}: z\in\Omega^i_{d_2}\},
\end{equation}
 where $\epsilon>0$ is the same as in Lemma \ref{l:estimates_v}. From the definition of $v_{z,\epsilon}$ it follows that $v_{z,\epsilon}\in\Huno$ for every $z\in\Omega^i_{d_2}$. Thus $V\subset N_{\mu}(\Omega)$. Furthermore,  since $\{I_{\mu}(\phi(t,u)): t\geq 0\}$ is bounded from below for each $v\in V$, we see that $\lim_{t\to\infty}I'_{\mu}(\phi(t,u))=0$ for each $v$ in $V$. Consequently, by Proposition \ref{p:PS}(i) there exists $\{\epsilon_t\}\subset\R^{+}$ such that $\lim_{t\to\infty}\epsilon_t=0$ and $\lim_{t\to\infty}||\phi(t,v)-u^{\mu}_{0,\epsilon_t}||=0$. This implies $$\lim_{t\to\infty}F(\phi(t,v))=0 \quad\forall\ v\in V.$$
Also invoking Lemma \ref{center of mass}, we have  $$F(\phi(t,v))\in\Omega_{d_1}.$$
As $\{F(\phi(0,v)): v\in V\}=\Omega^i_{d_2}$ we get that $\Omega^i_{d_2}$ is contractible in $\Omega_{d_1}$. But by the choice of $d_1$ and $d_2$ we have $\Omega_{d_1}\cong\Omega\cong\Omega^i_{d_2}$ which contradicts the assumption that $\Omega$ is not contractible. Hence equation \eqref{eq:problem1} has a solution.
\hfill{$\square$}

\vspace{3mm}

{\bf Acknowledgement} 

\vspace{2mm}

The main part of this research work was done when the author was a postdoc in University of New England, Australia. The author acknowledges the support of the  Australian Research Council (ARC). This research is partially supported by INSPIRE research grant DST/INSPIRE 04/2013/000152.  The author also expresses her sincere gratitude to the anonymous referee for his/her many valuable comments and suggestions which has helped to improve the manuscript in a great extent.

\label{References}

\end{document}